\documentclass{amsart}
\usepackage{amssymb,amscd}

 \newtheorem{thm}{Theorem}[section]
 \newtheorem{thmA}{Theorem}
 \newtheorem{cor}{Corollary}
 \newtheorem{lemma}[thm]{Lemma}
 \newtheorem{prop}[thm]{Proposition}
 \newtheorem{propA}[thmA]{Proposition}
 \theoremstyle{definition}
 \newtheorem{defn}[thm]{Definition}
 \newtheorem{exmp}[thm]{Example}
\newtheorem{notation}[thm]{Notation}
 \theoremstyle{remark}
 \newtheorem{rem}[thm]{Remark}

 \newtheorem{conj}[thm]{Conjecture}
 
 \numberwithin{equation}{subsection}

\newcommand{\FF}{\text{$\mathcal{F}$}}

\newcommand{\NN}{\text{$\mathcal{N}$}}

\newcommand{\ks}{{\text{\rm ks}}}

\newcommand{\intr}{\operatorname{int}}

\newcommand{\Lip}{\operatorname{Lip}}

        \newcommand{\field}[1]{\text{$\mathbb{#1}$}}
        \newcommand{\N}{\field{N}}
        \newcommand{\Z}{\field{Z}}
        
        \newcommand{\R}{\field{R}}
        \newcommand{\C}{\field{C}}

\newdimen\theight
\def\TeXref#1{%
             \leavevmode\vadjust{\setbox0=\hbox{{\tt
                     \quad\quad  {\small \textrm #1}}}%
             \theight=\ht0
             \advance\theight by \lineskip
             \kern -\theight \vbox to
             \theight{\rightline{\rlap{\box0}}%
             \vss}%
             }}%

\begin{document}

\title{Exotic open $4$-manifolds which are non-leaves}

\author{Carlos Meni\~no Cot\'on}
\author{Paul A. Schweitzer, S.J.}

\address{Departamento de An\'alise, Instituto de Matem\'atica e Estat\'istica\\
	Universidade Federal Fluminense\\
	Mario Santos Braga S/N, Niteroi, Rio de Janeiro 21941-916, Brazil.}

\email{carlos.meninho@gmail.com}

\address{Departamento de Matematica\\
         Pontificia Universidade Cat\'olica do Rio de Janeiro\\
         Marques de S\~ao Vicente 225, Gavea, Rio de Janeiro 22453-900, Brazil. }

\email{paul37sj@gmail.com}

\thanks{The first author wants to thank the Fundaci\'on Pedro Barri\'e de la Maza, Postdoctoral fellow (2013--2015) and the CAPES postdoc program (2015--2016) for their support.} 
\thanks{This work was also supported by MICINN, Grant MTM2014-56950-P Spain (2014--2017)}

\begin{abstract}
We study the possibility of realizing exotic smooth structures on punctured simply connected $4$-manifolds as leaves of a codimension one foliation on a compact manifold. In particular, we show the existence of uncountably many smooth open $4$-manifolds which are not diffeomorphic to any leaf of a codimension one transversely  $C^{2}$ foliation on a compact manifold. These examples include some exotic $\R^4$'s and exotic cylinders $S^3\times\R$.
\end{abstract}

\maketitle

\section*{Introduction}
The stunning results of Donaldson~\cite{Donaldson} and Freedman~\cite{Freedman} provided the existence of exotic smooth structures on $\R^4$, which is known to be the unique euclidean space with this property. This is in fact also true \cite{Bizaca-Etnyre} for an open $4$-manifold with a collarable end. The fact that these structures can arise in $4$-dimensional manifolds has implications for physics (see e.g. \cite{AsselmeyerMaluga-Brans,Krol}), i.e., what if our space-time carries an exotic structure? Since the exotic family was discovered in the $1980s$, nobody has been able to find an explicit and useful exotic atlas. It is worthy of interest to obtain alternative explicit descriptions of these exotica.

An open manifold which is realizable as a leaf of a foliation in a compact manifold must satisfy some restrictions. Since the ambient manifold is compact, an open manifold has to accumulate somewhere, and this induces recurrence and ``some periodicity'' on its ends.

Before reviewing the history of realizability of open manifolds as leaves,
we now state our main results.
In Section \ref{s:Yinfty} we shall define a class $\mathcal Y$ of smooth open manifolds (up to diffeomorphism) whose underlying topological manifolds are obtained by removing a finite non-zero number of
points from a closed, connected, simply connected topological $4$-manifold. In fact (see Example~\ref{exYinfty} and Remark \ref{r:continuum}), every such topological manifold is homeomorphic to uncountably many elements of $\mathcal Y$.

\begin{thmA}
No manifold $Y\in\mathcal{Y}$ is diffeomorphic to any leaf of a  $C^{2}$ codimension one foliation of a compact manifold.
\label{mainthm}\end{thmA}

Theorem~\ref{mainthm} and all of our results and proofs hold for the slightly weaker assumption of $C^{1+\Lip}$ regularity. For the sake of readability and coherence with the references we have decided to state this theorem for $C^2$ foliations. As a consequence of Theorem \ref{mainthm},
for every $Y\in\mathcal{Y}$
there are uncountably many diffeomorphically distinct smooth manifolds homeomorphic to $Y$
that cannot be leaves in any $C^2$ foliation of a compact $5$-manifold.
Also note that if $Z$ is obtained by puncturing a {\em smooth\/} closed simply connected manifold $M$, then the induced smooth manifold can easily be realized as a leaf of a $C^\infty$ codimension one foliation; just insert Reeb components along transverse closed curves in the product foliation of $M\times S^1$. 

It is unknown whether any element of $\mathcal{Y}$ is diffeomorphic to a leaf of a $C^{1,0}$ codimension one foliation on a compact manifold. If that can happen, we show that some restrictions must appear as to how it is realized; this is summarized in the next Proposition which is the main step in proving Theorem \ref{mainthm}.

\begin{propA}
If there exists a leaf diffeomorphic to $Y\in \mathcal{Y}$ in a $C^{1,0}$ codimension one
foliation of a closed $5$-manifold, then it is a proper leaf and
each connected component of the union of the leaves diffeomorphic
to $Y$ fibers over the circle with the leaves as fibers.
\label{prop1}\end{propA}

Next we review some of the history of leaves and non-leaves.
It was shown by J. Cantwell and L. Conlon \cite{Cantwell-Conlon} that every orientable open surface is homeomorphic (in fact, diffeomorphic) to a leaf of a foliation on each closed $3$-manifold, and non-orientable open surfaces are homeomorphic to leaves in non-orientable $3$-manifolds. The first examples of topological non-leaves were due to E. Ghys~\cite{Ghys1} and  T. Inaba, T. Nishimori, M. Takamura, N. Tsuchiya \cite{JAP}; these are highly topologically non-periodic open $3$-manifolds which cannot be homeomorphic to leaves in a codimension one foliation in a compact manifold. Years later, O. Attie and S. Hurder~\cite{Attie-Hurder}, in a deep analysis of the question, found simply connected $6$-dimensional examples of non-leaves, non-leaves which are homotopy equivalent to leaves and even a Riemannian manifold which is not quasi-isometric to a leaf in arbitrary codimension. These examples follow the line of the work of A. Phillips and D. Sullivan \cite{Phillips-Sullivan} and T. Januszkiewicz~\cite{Janusz} and led to other examples of Zeghib \cite{Zeghib} and the second author \cite{Schweitzer}. 

C.L. Taubes~\cite{Taubes} showed that the smooth structure of some of the exotic $\R^4$'s is, in some sense, non-periodic at infinity, and this leads to the existence of uncountably many non-diffeomorphic smooth structures on $\R^4$. It is an open problem whether any exotic $\R^4$---and, by extension, any given open manifold with a similar exotic smooth end structure---can be diffeomorphic to a leaf of a foliation on a compact manifold. By a simple cardinality argument, most exotic $\R^4$'s cannot be covering spaces of closed smooth $4$-manifolds by smooth covering maps since the diffeomorphism classes of smooth closed manifolds are countable.
All these results motivated a folklore conjecture in foliation theory suggesting that these exotic structures cannot occur on leaves of a smooth foliation in a compact manifold.

The main difference between some exotic $\R^4$'s (called {\em large}) and the standard $\R^4$ is the fact that they cannot embed smoothly in a standard $\R^4$. An important question for a large exotic $\R^4$ is to describe what are the simplest spin manifolds (in the sense of the second Betti number) in which it can be embedded; this is measured by the invariant defined by L. Taylor \cite{Taylor1}, which provided the first direct tool to show that some exotic $\R^4$'s cannot be non-trivial covering spaces. We shall show that these exotica are also non-leaves.

For all the other finitely many punctured simply connected $4$-manifolds we shall see that a {\em Taubes like\/} end (see Definition \ref{d:RS}) suffices to show that they are non-leaves. In  Proposition \ref{prop1} we adapt Ghys' procedure in \cite{Ghys1} to show some necessary conditions for such structures to be leaves of a codimension one foliation on a compact manifold.  In Theorem \ref{mainthm}, which is an easy corollary of Proposition \ref{prop1}, we complete this analysis in the case of  $C^{2}$ foliations (those where the transverse coordinate changes are  $C^{2}$ maps). 

The paper is organized as follows:

\begin{itemize}
\item The first section is devoted to exotic structures on open $4$-manifolds, particularly on $\R^4$. This is in fact a brief exposition of results in \cite{Taubes,Taylor1}. Here we define the particular exotic structures that we consider on $\R^4$ and show some of their properties.

\item In the second section we prove Proposition \ref{prop1} which gives necessary conditions for certain exotic punctured simply connected $4$-manifolds to be diffeomorphic to leaves, following Ghys' method of proof \cite{Ghys1}, and we derive its corollary Theorem \ref{mainthm}.

\item The last section includes some last remarks and open questions.
\end{itemize}

We would like to thank L. Conlon, R. Gompf, G. Hector, L. Taylor, and the referees for their help in preparing this paper. We also want to thank PUC-Rio and UFRJ as the main host institutions for their material support during the process of developing this work.

\section{Exotic structures on $\R^4$}

In this section we construct uncountably many exotic structures in $\R^4$ which are non-periodic by Taubes' work. Later we shall need a better control of this structure, which is provided by the invariant defined by L. Taylor \cite{Taylor1}. This introduction begins with a brief reminder of some known facts in $4$-dimensional topology.

\begin{thm}[Freedman \cite{Freedman}]\label{t:freedman}
Two simply connected closed $4$-manifolds are homeomorphic if and only if their intersection forms are isomorphic and they have the same modulo $2$ Kirby-Siebenmann invariant. In particular, simply connected smooth closed $4$-manifolds are homeomorphic if and only if their intersection forms are isomorphic.
\end{thm}

\begin{thm}[Donaldson \cite{Donaldson}]\label{t:Donaldson}
If a smooth closed simply connected $4$-manifold has a definite intersection form then it is isomorphic to a diagonal form.
\end{thm}

Definite symmetric bilinear unimodular forms are not classified and it is known that the number of
isomorphism classes grows at least exponentially with the range. Indefinite unimodular forms are classified \cite{Serre}: two indefinite forms are isomorphic if they have the same range, signature, and parity. There are canonical representatives for the indefinite forms; in the odd case the form is diagonal and in the even case it splits into invariant subspaces where the intersection form is either $\pm E_8$ or $H$. These canonical representatives are denoted as usual with the notation $m[+1]\oplus n[-1]$ for the odd case and $\pm mE_8\oplus nH$ with $n>0$ for the even one.

$$
E_8=\left( \begin{array}{cccccccc}
2 & 1 & 0 & 0 & 0 & 0 & 0 & 0\\
1 & 2 & 1 & 0 & 0 & 0 & 0 & 0\\
0 & 1 & 2 & 1 & 0 & 0 & 0 & 0\\
0 & 0 & 1 & 2 & 1 & 0 & 0 & 0\\
0 & 0 & 0 & 1 & 2 & 1 & 0 & 1\\
0 & 0 & 0 & 0 & 1 & 2 & 1 & 0\\
0 & 0 & 0 & 0 & 0 & 1 & 2 & 0\\
0 & 0 & 0 & 0 & 1 & 0 & 0 & 2
\end{array} \right);\quad
H=\left(\begin{array}{cc}
0 & 1\\
1 & 0
\end{array}\right)
$$ \bigskip

For each symmetric bilinear unimodular form  there exists at least one topological simply connected closed $4$-manifold with an isomorphic intersection form. But this is no longer true for the smooth case, as Donaldson's theorem asserts. It is an open problem what unimodular forms can be realized in smooth simply connected closed $4$-manifolds. Recall that a (not necessarily closed) simply connected smooth $4$-manifold is spin if and only if its intersection form is even. It is known that for a smooth simply connected spin $4$-manifold with indefinite intersection form the number of ``$E_8$ blocks'' must be even (Rokhlin's theorem). It is possible to say more, as in Furuta's theorem~\cite{Furuta} which will be useful in this section.

\begin{thm}[Furuta \cite{Furuta}]\label{t:Furuta}
If $M$ is a smooth closed spin (not necessarily simply-connected) 4-manifold
with an intersection form equivalent to $\pm 2mE_8\oplus nH$ and $m> 0$, then $n\geq 2m + 1$.
\end{thm}

Let us recall an important theorem of M.H. Freedman which is the main tool to determine when a manifold is homeomorphic to $\R^4$.

\begin{thm}[Freedman \cite{Freedman}]\label{t:R4}
An open $4$-manifold is homeomorphic to $\R^4$ if and only if it is contractible and simply connected at infinity.
\end{thm}

\begin{defn}
Two ends ${\bf e}_1$ and ${\bf e}_2$ of smooth manifolds are {\em diffeomorphic} if they have diffeomorphic neighborhoods
${\bf X}_{{\bf e}_1}$ and ${\bf X}_{{\bf e}_2}$. It will always be assumed that orientation is preserved by that diffeomorphism. Two manifolds with one end are {\em end-diffeomorphic} if their ends are diffeormorphic
\end{defn}

The main tool for measuring the wildness of some exotica will be the {\em Taylor index\/}, introduced by L. Taylor in \cite{Taylor1}.

\begin{defn}[Taylor \cite{Taylor1}]
Let $E$ be a smoothing of $\R^4$. Let $\mathit{Sp}(E)$ be the set of closed smooth spin $4$-manifolds $N$ with trivial or hyperbolic intersection form (a sum of copies of $H$)  in which $E$ embeds smoothly. Define $b_E=\infty$ if $\mathit{Sp}(E)=\emptyset$, or else:
$$
2b_E=\min_{N\in\mathit{Sp}(E)}\{b_2(N)\}\;$$
where $b_2(N)$ is the second Betti number of $N$.

Let $\mathcal{E}(E)$ be the set of topological embeddings $e:D^4\to E$ such that $e$ is smooth in the neighborhood of some point of the boundary and $e(\partial D^4)$ is (topologically) bicollared. Set $b_e=b_{e(\mathring{D}^4)}$ where $e(\mathring{D^4})$ has the smooth structure induced by $E$. The {\em Taylor index\/} of $E$ is defined to be
$$
\gamma(E)=\max_{e\in\mathcal{E}(E)}\{b_e\}\;.
$$
For a spin manifold $M$, the Taylor index of $M$ is the supremum of the Taylor-indices of all the exotic $\R^4$'s embedded in $M$.
\label{d:Taylor index}\end{defn}

Another important tool for this section is the ``end sum'' construction. For open manifolds this is analogous to the connected sum of closed manifolds. Given two open smooth oriented manifolds $M$ and $N$ with the same dimension we choose two smooth properly embedded  paths $c_1:[0,\infty)\to M$ and $c_2:[0,\infty)\to N$, each of them defining one end in $M$ and $N$ respectively. Let $V_1$ and $V_2$ be smooth tubular neighborhoods of $c_1([0,\infty))$ and $c_2([0,\infty))$. The boundaries of these neighborhoods are clearly diffeomorphic to $\R^3$ and we can obtain a smooth sum by removing the interiors of these neighborhoods and identifying their boundaries so as to produce a manifold with an orientation respecting the orientations of $M$ and $N$. This will be called the {\em end sum} of $M$ and $N$ associated to $c_1$ and $c_2$, and it is denoted by $M\natural N = (M\setminus \mathring{V_1})\bigcup_\partial (N\setminus \mathring{V_2})$. In the case where $N$ and $M$ have exactly one end and are both homeomorphic to $S^3\times\R^+$,  $c_1$ and $c_2$ are unique up to ambient isotopy and thus the smooth structure of $M\natural N$ does not depend on the chosen paths.
End sum was the first technique which made it possible to find infinitely many exotic structures on $\R^4$~\cite{Gompf2} and it is an important tool for dealing with the problem of generating infinitely many smooth structures on open $4$-manifolds~\cite{Bizaca-Etnyre,Ganzell}.

\begin{lemma}[Lemma 5.2 \cite{Taylor1}]\label{l:subadditivity}
If $\mathbf{R}$ and $\mathbf{S}$ are exotic $\R^4$'s, then $\gamma(\mathbf{R}\natural\mathbf{S})\leq\gamma(\mathbf{R})+\gamma(\mathbf{S})$.
\end{lemma}

We now present a version of Taubes' theorem which suffices for our purposes.

\begin{defn}[Periodic end]\label{d:periodic_end}
Let $M$ be an open smooth manifold with an end homeomorphic to $S^3\times (0,\infty)$. We say that this end is {\em smoothly periodic\/} if there exists a neighborhood $V\subset M$ of the end that is homeomorphic to $S^3\times (0,\infty)$ and a diffeomorphism $h:V\to h(V)\subset V$ such that $h^n(V)$ defines the given end (i.e., $\{h^n(V)\}$ is a neighborhood base for the end).
\end{defn}

Note that this notion of smoothly periodic end is a particular case of the admissible periodic ends considered in \cite{Taubes}.

\begin{thm}[Taubes \cite{Taubes}] \label{t:Taubes}
Let $M$ be an open smooth simply connected $4$-manifold with a definite intersection form and exactly one end. If the end of $M$ is homeomorphic to $S^3\times(0,\infty)$ and smoothly periodic, then the intersection form is isomorphic to a diagonal form. 
\end{thm}

\begin{defn}\label{d:RS}
Throughout this work $\mathcal{M}_-$ (resp. $\mathcal{M}_+$) will denote the family of smoothings of closed topological $4$-manifolds, $M$, with exactly one puncture so that there exists $s\in\N$ such that $\natural_{i=1}^s M$ is end-diffeomorphic to a smoothing of a punctured topological simply connected negative (resp. positive) definite but not diagonal $4$-manifold. Set $\mathcal{M}=\mathcal{M}_-\cup\mathcal{M}_+$. These manifolds and their ends will be called {\em Taubes like\/}.

The set $\mathcal{S}$ will denote the family of all exotic $\R^4$'s $\mathbf{R}$ for which there exist $s,k\in\N^+$ such that the $s$-fold end-sum $\natural^s_{i=1}\mathbf{R}$
is end-diffeomorphic to a smoothing of a punctured simply connected spin $-k(E_8\oplus E_8)$ manifold\footnote{These are called $(s,k)$-{\em simple-semi-definite\/} in \cite{Taylor1}.}. The subfamily $\mathcal{R}$ will be formed by the exotica in $\mathcal{S}$ with finite Taylor index.
\end{defn}

\begin{rem}\label{r:prop S-R}
Of course, $\mathcal{R}\subset\mathcal{S}\subset\mathcal{M}_-$. Observe that $(\mathcal{M}_{\pm},\natural)$, $(\mathcal{S},\natural)$ and $(\mathcal{R},\natural)$ are semigroups, this comes from the fact that the sum of non-diagonal definite forms of the same sign is still definite and non-diagonal by the Eichler-Kneser Theorem (see e.g. Theorem 9.24 in \cite{Gerstein}). Observe also that infinite end sums of elements in $\mathcal{S}$ are still elements of $\mathcal{S}$. We shall see that this property is false for $\mathcal{R}$ (see Proposition \ref{p:TaylorR}; of course it is also false for $\mathcal{M_\pm}$ just by purely topological reasons).

Let $M$ and $N$ be two smooth $4$-manifolds end-homeomorphic to $\R^4$ so that $M\subset N$ and $M\setminus N$ is homeomorphic to $S^3\times [0,\infty)$ with topologically bicollared boundary. If $M\in\mathcal{M}$ then $N\in\mathcal{M}$.
\end{rem}

\begin{rem}\label{r:TaubesM} No end of a Taubes like manifold is smoothly periodic: if $M\in\mathcal{M}$ is smoothly periodic then $\natural^s M$ will be also smoothly periodic for all $s\in\N$. But for some $s$ this manifold would be end-diffeomorphic to a definite non-diagonal simply connected $4$-manifold (by definition of $\mathcal{M}$), but this is not possible by Taubes' Theorem.  In particular, Taubes like ends are exotic, i.e., they are not diffeomorphic to the ends of the standard $S^3\times \mathbb R$, which are obviously smoothly periodic.

\end{rem}

\begin{exmp}\label{e:known examples}
The existence of exotica in $\mathcal{R}$ and $\mathcal{S}$ with these properties is well known (see e.g. \cite{Ganzell, Gompf-Stipsicz, Taylor1}). Let $M_0$ be the $K3$ Kummer surface. It is known that the intersection form of $M_0$ can be written as $-2E_8\oplus3H$, where the six elements in $H_2(M_0,\Z)$ spanning the summand $3H$ can be represented by six Casson handles $C_i$ attached to a closed $4$-dimensional ball $B^4$ inside $M_0$. Let $U=\intr(B^4\cup \bigcup_{i=1}^6C_i)$ which is clearly homeomorphic to a punctured $\#^3S^2\times S^2$ by Freedman's theorem~\ref{t:freedman}. Let $S$ be the union of the cores of the Casson handles, which we consider to be inside $\#^3S^2\times S^2$. By Theorem~\ref{t:R4} the manifold $\mathbf{P}=\#^3S^2\times S^2\setminus S$ is homeomorphic to $\R^4$. If this $\mathbf{P}$ were standard then we could smoothly replace the $3H$ part in the intersection form of $M_0$ by a standard ball, so the resulting smooth closed manifold would have intersection form $-(E_8\oplus E_8)$, in contradiction to Donaldson's theorem (Theorem~\ref{t:Donaldson}), since $(-E_8\oplus E_8)$ is not isomorphic to a diagonal form. Since $\mathbf{P}$  is contained in $\#^3S^2\times S^2$ it follows that $\gamma(\mathbf{P})\leq 3$ and therefore $\mathbf{P}\in\mathcal{R}$.\end{exmp}

\begin{exmp}\label{e:T}
With some care in the above construction (see e.g. Example 5.6 in \cite{Taylor1}) the six Casson handles can be arranged as three diffeomorphic pairs that can each be embedded in $S^2\times S^2$. As above, the complement of a neighborhood of each core in $S^2\times S^2$ is an exotic $\R^4$, say $\mathbf{T}$, such that $\natural^3\mathbf{T}=\mathbf{P}$. It is clear that $\mathbf{T}$ cannot be end-diffeomorphic to any spin punctured $mE_8$ manifold with $m\in\Z, m\neq 0$, for otherwise we would obtain a smoothable closed spin simply connected $(mE_8)\oplus H$ manifold: just use the end of $\mathbf{T}$ to attach $S^2\times S^2\setminus\mathbf{T}$ to that punctured manifold. But this is impossible by Furuta's and Rokhlin's theorems (depending on whether $m$ is even or odd). This is another large exotic $\R^4$ which belongs to $\mathcal{R}$.\end{exmp}

\begin{exmp}\label{e:Eetc}
Another interesting exotic $\R^4$ is given as follows. Take the indefinite form $-(E_8\oplus E_8)\oplus\langle 1\rangle$. It follows by Freedman's Theorem \ref{t:freedman} that there exists a closed simply connected manifold with such an intersection form that is homeomorphic to $\#^{16}\overline{\C P^2}\#\C P^2$ (where $\overline{\C P^2}$ denotes the complex projective plane with the opposite orientation). Thus the former manifold is smoothable but it is impossible to represent this manifold as a connected sum of a $-(E_8\oplus E_8)$ manifold with $\C P^2$ using a smooth $3$-sphere. But the homology generator of the $\C P^2$ summand can be represented by a Casson handle. As above, that Casson handle can be embedded in $\C P^2$ representing its homology. Removing a suitable neighborhood of its core we get an exotic $\R^4$, that will be called $\mathbf{E}$, which is end-diffeomorphic to a punctured spin simply connected $-(E_8\oplus E_8)$ manifold and can be embedded in $\C P^2$ (see details in Example 5.10 of \cite{Taylor1}). It is unknown whether its Taylor index is finite. In particular it is unknown if $\mathbf{E}$ is diffeomorphic to $\mathbf{P}$.
\end{exmp}

\begin{exmp}
The family $\mathcal{S}$ includes the universal exotic $\R^4$ presented in \cite{Freedman-Taylor} and any other possible universal smoothing of $\R^4$ (see the last paragraph in Remark~\ref{r:prop S-R}).
\end{exmp}

\begin{exmp}
Finally, a smoothing of a punctured $\pm E_8$ manifold is trivially a Taubes like manifold, so its end is not standard. But it is unknown whether this manifold is end-diffeomorphic to an exotic $\R^4$.
\end{exmp}

\begin{notation}\label{disk K_t}
Let $\mathbf{R}$ be an exotic $\R^4$ and let $\psi_{\mathbf{R}}:\R^4\to\mathbf{R}$ be a homeomorphism. Let us denote $\mathbf{K}^{\psi_\mathbf{R}}_t=\psi_\mathbf{R}(D(0,t))$, where $D(0,t)$ is the standard closed disk of radius $t$, and consider the smooth structure induced
on $\mathring{\mathbf{K}}^{\psi_\mathbf{R}}_t$ by $\mathbf{R}$. For future reference, we choose the homeomorphism $\psi_\mathbf{R}$ so that for each $t>0$ the boundary of the topological disk $\mathbf{K}_t$ is smooth in a neighborhood of some point. The existence of such $\psi_{\mathbf{R}}$'s is clear; in fact, they can be chosen to be smooth in a neighborhood of one of the axes \cite{Quinn}. By an abuse of notation, we shall use the notation $\mathbf{K}_t$ instead of $\mathbf{K}^{\psi_\mathbf{R}}_t$ whenever the underlying exotic $\R^4$ is clear from the context and $\psi_\mathbf{R}$ is any homeomorphism as above.

\end{notation}

\begin{cor}[Theorem 5.4 in \cite{Taylor1}]\label{c:TaubesS}
Let $\mathbf{R}\in\mathcal{S}$ and let $\psi:\R^4\to\mathbf{R}$ be a homeomorphism. Then there exists $r>0$ such that, for any $t>s>r$, $\mathring{\mathbf{K}}_t$ is not diffeomorphic to  $\mathring{\mathbf{K}}_s$.
\end{cor}

In the following, whenever $\mathbf{R}\in\mathcal{S}$ (resp. $\mathcal{R}$), $r_{\psi_\mathbf{R}}$ will be assumed large enough so that $\mathring{\mathbf{K}}_t\in\mathcal{S}$ (resp. $\mathcal{R}$) for all $t>r_{\psi_\mathbf{R}}$.

If $N_1$ and $N_2$ are oriented manifolds with connected boundaries, where each boundary is assumed to be smooth in a neighborhood of some point, then the {\em boundary connected sum} $M=N_1\#_\partial N_2$ is obtained by 
identifying embedded smooth closed disks in $\partial N_1$ and $\partial N_2$ by an orientation reversing diffeomorphism (so that the resulting manifold is oriented) and smoothing the result. Then the interior of the boundary connected sum is the end sum of the interiors, $\mathring{M}=\mathring{N_1}\natural\mathring{N_2}$. If $N_1$ and $N_2$ are disjoint connected oriented codimension zero submanifolds with connected boundaries embedded in a connected oriented manifold $M$ so as to respect the orientations, then $N_1\#_\partial N_2$ can also be embedded in $M$ using  a standard cylinder to join smooth standard disks in the boundaries of $N_1$ and $N_2$.

\begin{prop}[Proposition 2.2 \cite{Taylor1}]\label{p:TaylorD}
Let $e_i\in\mathcal{E}(E)$, $i=1,\dots,k$, be pairwise disjoint closed topological disks in $E$  (respecting the orientations). Then there exists $e\in\mathcal{E}(E)$ such that $e(\mathring{D^4})$ is diffeomorphic to $\natural_{i=1}^k e_i(\mathring{D^4})$.
\end{prop}

\begin{prop}[Theorem 5.3 \cite{Taylor1}]\label{p:TaylorR}
Let $\mathbf{R}\in\mathcal{S}$ and let $s,k\in\N^+$ so that $\natural^s\mathbf{R}$ is end diffeomorphic to a spin simply connected punctured $-k(E_8\oplus E_8)$ manifold.  Then $0<2k/s<\gamma(\mathbf{R})$. Therefore $\gamma(\natural^n \mathbf{R})$ tends to $\infty$ as $n\to\infty$, and so $\gamma(\natural_{i=1}^\infty\mathbf{R})=\infty$.
\end{prop}

By the choice of $\psi_\mathbf{R}$, it follows that
$\mathbf{K}_t\in\mathcal{E}(\mathbf{R})$ (i.e., $\mathbf{K}_t$ is a topologically embedded ball with the properties indicated in Definition \ref{d:Taylor index}). The next Proposition can be seen as a Corollary of Theorem 7.1 in \cite{Taylor1}. 

\begin{prop}[Theorem 7.1 \cite{Taylor1}]
No $\mathbf{R}\in\mathcal{R}$ can be a smooth covering space of a smooth compact $4$-manifold. 
\end{prop}
\begin{proof}
If this were the case then there would exist a properly discontinuous smooth $\Z$-action on $\mathbf{R}$. It follows that $\mathbf{R}$ would contain infinitely many pairwise disjoint copies of sets diffeomorphic to $\mathbf{K}_t$ for any $t>r_{\psi_\mathbf{R}}$. By Proposition~\ref{p:TaylorD}, it follows that $\natural^\infty \mathring{\mathbf{K}}_t$ could be embedded in $\mathbf{R}$. Since $r_{\psi_\mathbf{R}}$ was chosen so that $\mathbf{K}_t$ belongs to $\mathcal{S}$, it follows that $\gamma(\mathbf{R})=\infty$ by Proposition~\ref{p:TaylorR}. But $\gamma(\mathbf{R})$ must be finite since it is an element of $\mathcal{R}$. 
\end{proof}

On the other hand, $\natural_{i=1}^\infty\mathbf{R}$ can be a non-trivial covering space of an open manifold. In fact it admits several free actions (see e.g. \cite{Gompf, Gompf-Stipsicz}); for example,
this exotic is diffeomorphic to the
end sum $\natural_{i\in\mathbb{Z}}\mathbf{R}$, which
admits an obvious free action of $\mathbb{Z}$ whose quotient is
an exotic $\R^3\times S^1$.

\begin{rem}\label{r:infinite index}
In Example 5.10 in \cite{Taylor1}, uncountably many non-diffeomorphic smooth structures on $\R^4$ with infinite Taylor index are exhibited. Consider the element $\mathbf{E}\in\mathcal{S}$ presented in Example~\ref{e:Eetc}. It is shown that the manifolds $\mathbf{\mathring{K}}^{\psi_\mathbf{E}}_t\natural\left(\natural_{i=1}^\infty\mathbf{E}\right)$ are pairwise non-diffeomorphic for all $t>r_{\psi_\mathbf{E}}$. Remark also that, although $\natural_{i=1}^\infty\mathbf{E}$ cannot be embedded in any spin closed manifold with hyperbolic intersection form, it can be embedded in $\C P^2$.
\end{rem}

\begin{rem}\label{r:good-compact} 
Let $\mathbf{R}\in\mathcal{S}$. Given a strictly increasing sequence $\{t_k\}$ tending to infinity with every $t_k>r_{\psi_{\mathbf{R}}}$,  let us set $\mathbf{C}^{\psi_\mathbf{R}}_{k}=\mathbf{K}_{t_k}\setminus \mathring{\mathbf{K}}_{t_{k-1}}$, each of which is homeomorphic to $S^3\times [0,1]$. We use the notation $\mathbf{C}_k$ instead of $\mathbf{C}_k^{\psi_\mathbf{R}}$ whenever $\mathbf{R}$ and $\psi_\mathbf{R}$ are clear from the context.

Of course, $\mathbf{C}_k$ also depends on the sequence $\{t_k\}$, but these data are inessential and will be omitted for the sake of simplicity. A similar notation can be adapted to the end of any Taubes like manifold $M$. Let $X$ be a cylindrical neighborhood of the end of $M\in\mathcal{M}$ and let $\psi_X:X\to S^3\times [0,\infty)$ be a homeomorphism. Given an increasing sequence of positive numbers $t_k$ going to infinity, we can set $\mathbf{C}_k=\psi_X^{-1}(S^3\times[t_{k-1},t_k])$ with the induced smooth structure as a subset of $M$. Again there is a dependence on the sequence, the neighborhood $X$, and the homeomorphism $\psi_X$ which will also be omitted.

For any $M\in\mathcal{M}$ with $\mathbf{C}_k$ smoothly embedded in $M$, $M\setminus\mathring{\mathbf{C}}_k$ has two components, one of them compact, say $K_k$, and the other unbounded, say $X_k$. Let us denote the boundary component of $\mathbf{C}_k$ which bounds $K_k$ by $\partial^-\mathbf{C}_k$ and the boundary component which bounds $X_k$ by $\partial^+\mathbf{C}_k$. We claim that for every sufficiently large $k$, $\mathring{K}_k$ is also a Taubes like manifold. There exists $s$ so that $\natural^s M$ is end-diffeomorphic to a definite non-diagonal open $4$-manifold, say $N$. Let $Y$ be a neighborhood of the end of $N$ which is homeomorphic to $S^3\times [0,\infty)$ and diffeomorphic (preserving orientation) to a neighborhood $Y'$ of the end of $\natural^s M$. For all suficiently large $k$ the set $\natural^s \mathring{K}_k\cap Y'$ is also homeomorphic to $S^3\times [0,\infty)$. Then $\natural^s \mathring{K}_k$ is end-diffeomorphic to $(N\setminus Y)\cup_\partial (\natural^s\mathring{K}_k\cap Y')$, which is also a simply connected definite non-diagonal manifold (homeomorphic to $N$), so $\mathring{K}_k$ is in fact Taubes like.
\end{rem}

\begin{rem}
Observe that all the above results given for $\mathcal{R}$ and $\mathcal{S}$ can be applied verbatim to $\overline{\mathcal{R}}$ and $\overline{\mathcal{S}}$ which denote the same manifolds but with the orientation reversed (so they are elements in $\mathcal M_+$). This follows from the fact that changing orientations does not affect Furuta's Theorem \ref{t:Furuta} or the Taylor index. \label{r:barR}
\end{rem}

\section{Exotic simply connected smooth $4$-manifolds and foliations}  \label{s:Yinfty} 

\subsection{The Family of Exotica}
Here we define certain exotic simply connected smooth manifolds that we shall show cannot be diffeomorphic to leaves of a $C^2$ codimension one foliation in a compact smooth manifold. As mentioned in the introduction, we are interested in the set of open $4$-manifolds which are obtained by removing a finite non-zero number of points from a closed, connected, simply connected topological $4$-manifold.

\begin{defn}
Let $\mathcal Y$ be the set of open smooth $4$-manifolds $Y$ (up to diffeomorphism) that are
homeomorphic to simply connected closed topological $4$-manifolds with finitely many punctures, such that
$Y$ satisfies one of the following two conditions (see Definition \ref{d:RS} and Remark \ref{r:barR}):
\begin{enumerate}

\item $Y\in\mathcal{R}\cup\overline{\mathcal{R}}$ (and so $Y$ is homeomorphic to $\R^4$);

\item $Y$ is not homeomorphic to $\R^4$ and at least one (exotic) end is diffeomor\-phic (preserving orientation) to the end of some element in $\mathcal{M}$ (i.e., it is a Taubes like end).
\end{enumerate}
\label{def} \end{defn}

\noindent Observe that $\mathcal{Y}$ is the (non-disjoint) union of two families: $\mathcal{Y}_f$ where at least one Taubes like end is diffeomorphic to the end of an exotic $\R^4$ with finite Taylor index and $\mathcal{Y}_\infty$ where at least one Taubes like end is not diffeomorphic to the end of an exotic $\R^4$ with finite Taylor index. By definition, $\mathcal{R}\cup\overline{\mathcal{R}}\subset\mathcal{Y}_f$, hence only $\mathcal{Y}_f$ contains exotic $\R^4$'s.

\begin{exmp}
If $Z$ is a simply connected smooth closed $4$-manifold that is not homeomorphic to $S^4$, then $Z\#\mathbf{R}$ belongs to ${\mathcal Y}$ for all $\mathbf{R}\in\mathcal{S}\cup\overline{\mathcal{S}}$. All these manifolds are homeomorphic but not diffeomorphic to the standard $Z\setminus\{\ast\}$ by Taubes' work.

If $Z$ is an arbitrary simply connected but non-smoothable closed $4$-manifold, after removing a point it becomes smoothable (every open $4$-manifold is smoothable, see e.g. \cite{Freedman-Quinn}). By Theorem 2.1 in \cite{Gompf}, $Z\setminus\{\ast\}$ admits a smoothing which is end-diffeomorphic to some simply connected punctured spin $-kE_8$ manifold (in fact, we can take $k=2$ if the Kirby-Siebenmann invariant of $Z$, $\ks(Z)$, is trivial and $k=3$ otherwise). So in both cases we get smooth Taubes like manifolds in $\mathcal Y$ homeomorphic to $Z\setminus\{\ast\}$. Moreover, when $\ks(Z)=0$ that end is shown to be end-diffeomorphic to an exotic $\R^4$ with finite Taylor index (it can be embedded in $\#^nS^2\times S^2$ for some sufficiently large $n$), so that smoothing belongs to $\mathcal{Y}_f$. In any case, after forming an infinite end sum with (possibly distinct) elements in $\mathcal{S}$ we also obtain smooth manifolds in $\mathcal{Y}_\infty$ homeomorphic to $Z\setminus\{\ast\}$.

Of course, we can add more punctures to these smoothings and we still get elements in $\mathcal{Y}$, so any topological simply connected $4$-manifold obtained by removing finitely many punctures from a closed manifold is homeomorphic to some element in $\mathcal{Y}$.
\label{exYinfty}  
\end{exmp}

\begin{rem}\label{r:continuum}
Taubes' theorem (Theorem \ref{t:Taubes}, Remark~\ref{r:TaubesM} and Corollary~\ref{c:TaubesS}) shows that for any $Y\in\mathcal{Y}$ there exists an uncountable family of smooth manifolds in $\mathcal{Y}$ which are homeomorphic but non-diffeomorphic to $Y$. The same argument works for elements in $\mathcal{Y}_f$, i.e., for any element in $\mathcal{Y}_f$ there is a continuum of elements in $\mathcal{Y}_f$ which are homeomorphic but not diffeomorphic to it.

Although Taubes' theorem applies also to elements in $\mathcal{Y}_\infty$ we cannot derive a continuum of structures in $\mathcal{Y}_\infty$. For instance, consider $\natural^\infty\mathbf{P}$ (see Example~\ref{e:known examples} for the definition of $\mathbf{P}$); it has infinite Taylor index by Proposition~\ref{p:TaylorR} but any $\mathring{\mathbf{K}}_t$ has finite Taylor index since it is contained in a finite end-sum of copies of $\mathbf{P}$ which has finite Taylor index by Lemma~\ref{l:subadditivity}. 

However Taylor's results (see Remark \ref{r:infinite index}) show that for all $Y\in\mathcal{Y}_\infty$ which are end-diffeomorphic to $\natural^\infty\mathbf{E}$ there is a continuum of manifolds in $\mathcal{Y}_\infty$ which are homeomorphic but non-diffeomorphic to $Y$ (this is also true if the orientation is reversed). Let $M$ be a topological (non-smoothable) simply connected compact manifold. If $\ks(M)=0$ then $M\# \C P^2$ or $M\# \overline{\C P^2}$ is smoothable by Freedman's Theorem~\ref{t:freedman} since the intersection form becomes indefinite. By surgering the $\C P^2$ or $\overline{\C P^2}$ summand by means of a suitable neighborhood of the core of a Casson handle (diffeomorphic to the given in Example~\ref{e:Eetc}, since any two Casson handles have a common refinement) we can obtain a smoothing of $M\setminus\{\ast\}$ end-diffeomorphic to $\mathbf{E}$ or $\overline{\mathbf{E}}$. It follows that every simply connected topological $4$-manifold with finite punctures and trivial modulo 2 Kirby-Siebenmann invariant admits a continuum of smoothings in $\mathcal{Y}_\infty$. It is not clear for us how to solve this question when $\ks(M)\neq 0$ and there is a single punture (for two or more puntures, just take an smoothing where one end is standard and then perform an end-sum with $\natural^\infty\mathbf{E}$).
\end{rem}

\subsection{The Proof of Proposition \ref{prop1}}
Recall that Proposition \ref{prop1} states that if a leaf of a $C^{1,0}$ codimension one foliation of a compact manifold is diffeomorphic to $Y\in \mathcal Y$, then the leaf is proper and each connected component of the union of all leaves diffeomorphic to $Y$ fibers over the circle with the leaves as fibers.

\begin{rem}
Regularity $C^{1,0}$ means that the leaves are $C^1$ and tangent to a continuous hyperplane distribution of codimension one.
\end{rem}

In proving Proposition \ref{prop1} we use the basic theory of codimension one foliations of smooth compact manifolds presented as integrable plane fields. Note that in this general situation there exists a smooth transverse one-dimensional foliation $\NN$ and a biregular foliated atlas, i.e., one in which each coordinate neighborhood is foliated simultaneously as a product by $\FF$ and $\NN$. The transverse coordinate changes are only assumed to be continuous but the leaves can be taken to be smooth manifolds and the local projection along $\NN$ of one plaque onto another plaque in the same chart is a diffeomorphism. Our basic tools are Dippolito's octopus decomposition and his semistability theorem \cite{Candel-Conlon,Dippolito} as well as the trivialization lemma of G. Hector \cite{Hector}. We assume that our foliation is transversely oriented, which is not a real restriction since every manifold considered to be a leaf is simply connected and therefore, by passing to the transversely oriented double cover, a transversely oriented foliation with a leaf diffeomorphic to it is obtained. 

For a saturated open set $U$ of $(M,\FF)$, let $\hat{U}$ be the completion of $U$ for a Riemannian metric of $M$ restricted to $U$. The inclusion $i:U\to M$ clearly extends to an immersion $i:\hat{U}\to M$, which is at most $2$-to-$1$ on the boundary leaves of $\hat{U}$. We shall use $\partial^\tau$ and $\partial^\pitchfork$ to denote the tangential and transverse boundaries, respectively.

\begin{thm}[Octopus decomposition \cite{Candel-Conlon,Dippolito}]\label{t:octopus}
Let $U$ be a connected saturated open set of a codimension one transversely orientable foliation $\FF$ in a compact manifold $M$.  There exists a compact submanifold $K$ (the nucleus) with boundary and corners such that
\begin{enumerate}
\item $\partial^\tau K \subset\partial^\tau\hat{U}$

\item $\partial^\pitchfork K$ is saturated for $i^\ast\NN$

\item the set $\hat{U}\setminus K$ is the union of finitely many non-compact connected components $B_1,\dots,B_m$ (the arms) with boundary, where each $B_i$ is diffeomorphic to a product $S_i\times[0,1]$ by a diffeomorphism $\phi_i:S_i\times[0,1]\to B_i$ such that the leaves of $i^\ast\NN$ exactly match the fibers $\phi_i(\{\ast\}\times[0,1])$.

\item the foliation $i^\ast\FF$ in each $B_i$ is defined as the suspension of a homomorphism from $\pi_1(S_i)$ to the group of homeomorphisms of $[0,1]$. Thus the holonomy in each arm of this decomposition is completely described by the action of $\pi_1(S_i)$ on a common complete transversal.
\end{enumerate}
\end{thm}

Observe that this decomposition is far from being canonical, for the compact set $K$ can be extended in many ways yielding other decompositions. We do not consider the transverse boundary of $B_i$ to be a part of $B_i$; in particular, the leaves of $i^\ast\FF_{|{B_i}}$ are open sets in leaves of $i^\ast\FF$. Remark also that the word diffeomorphism will only be applied to open sets (of $M$ or of leaves of $\FF$); on the transverse boundaries the maps are only considered to be homeomorphisms.

\begin{lemma}[Trivialization Lemma \cite{Hector}]\label{l:trivialization}
Let $J$ be an arc in a leaf of $\NN$. Assume that each leaf meets $J$ in at most one point. Then the saturation of $J$ is diffeomorphic to $L\times J$, where $L$ is a leaf of $\FF$, and the diffeomorphism carries the bifoliation $\FF$ and $\NN$ to the product bifoliation of $L\times J$ (with leaves $L\times\{\ast\}$ and $\{\ast\}\times J$).
\end{lemma}

\begin{thm}[Dippolito semistability theorem \cite{Candel-Conlon,Dippolito}]\label{t:semistability}
Let $L$ be a semiproper leaf which is semistable on the proper side, i.e., there exists a sequence of fixed points for all the holonomy maps of $L$ converging to $L$ on the proper side. Then there exists a sequence of leaves $L_n$ converging to $L$ on the proper side and projecting diffeomorphically onto $L$ via the fibration defined by $\NN$.
\end{thm}

\begin{notation}\label{n:K Y}
Let $X$ be a neighborhood of the ends of $Y\in {\mathcal Y}$ identified (topologically) with $\bigsqcup_{i=1}^n S^3\times [0,\infty)$ such that the boundaries $\bigsqcup_{i=1}^n S^3\times
\{0\}$ are (topologically) bicollared in $Y$. Then
we have the decomposition
$$Y=K_Y \cup X$$
where $K_Y$ is
 $Y\setminus \mathring{X}$, so it is compact with boundary, and, in the case that $Y$ is not homeomorphic to $\R^4$ with finite punctures, it has non-trivial second homology by Freedman's Theorem \ref{t:freedman}, since removing a finite number of points does not change the second homology. Since the boundary is a disjoint union of (topological) $3$-spheres, it also follows that $H_2(\partial K_Y)=0$.
\end{notation}

Now we have enough information to begin to follow the line of reasoning of Ghys \cite{Ghys1}. For the rest of this section we assume that $Y\in {\mathcal Y}$ is diffeomorphic to a leaf, and we shall find some constraints.

\begin{defn}\label{d:vanishing_cycle}
We say that a leaf $L\in\FF$ {\em contains a  lacunary vanishing cycle\/} if there exists a topologically bicollared connected embedded closed oriented
$3$-manifold $\Sigma\subset L$ and a family of connected $3$-manifolds $\{\Sigma(n)\  |\ n\in\N\}$ embedded in the same leaf $L$ that are 
null-homologous on $L$ and converge to $\Sigma$ along leaves of the transverse foliation $\NN$. 
It is a {\em trivial lacunary vanishing cycle} if $\Sigma$ is null-homologous on $L$.
\end{defn}

\begin{lemma}\label{l:trivial lacunary}
Let $L$ be a simply connected leaf with an end $\mathbf{e}$ homeomorphic to $S^3\times (0,\infty)$. If $L\subset\lim_{\mathbf{e}}(L)$ then $L$ does not contain any non-trivial lacunary vanishing cycle homeomorphic to $S^3$.
\end{lemma}

To prove this, we shall use a special case of a weak generalization of Novikov's theorem on the
existence of Reeb components, Theorem 4 of \cite{Schweitzer}. Recall that a (generalized) Reeb component with connected boundary is a compact $(k+1)$-manifold with a codimension one foliation such that the boundary is a leaf and the interior fibers over the circle with the leaves as fibers. 

Suppose we are given a
compact $(k+1)$-manifold $M$ with a transversely oriented codimension one foliation $\FF$, a transverse one-dimensional foliation $\NN$,
a closed connected $(k-1)$-manifold $B$ 
and also a bifoliated map $h: B\times [a,b]\to M$, where $[a,b]$ is an interval in the real line,
such that $h(B\times\{t\})$ is contained in a leaf $L_t$ of $\FF$ for every $t\in [a,b]$ and $h_a:B\to L_a$ is an embedding
with bicollared image, where $h_t(x)=h(x,t)$. Since $h$ is bifoliated, it follows that $h_t$  is an embedding for all $t$.

\begin{thm}[See \cite{Schweitzer}, Theorem 2.13 (2)] If for every $t\in(a,b]$, $B_t=h_t(B)$ bounds a compact
connected region in $L_t$, but $B_a$ does not bound on $L_a$,
then $L_a$ is the boundary of a Reeb component whose interior leaves are the leaves $L_t$ for $t\in(a,b]$. \label{Reebcomponent}
\end{thm}

\begin{proof} [Proof of Lemma~\ref{l:trivial lacunary}]
Suppose that $L\subset\lim_{\mathbf{e}}(L)$ and that
$\Sigma\subset L$ is a non-trivial lacunary vanishing cycle on $L$ that is homeomorphic to $S^3$.
Since $\FF$ is assumed to be transversely oriented, the transverse foliation $\NN$ defines a  map 

$$\Phi:S^3\times\R\to M\ $$ 
that takes each set $S^3 \times \{t\}$ into a leaf of $\FF$ and such that $\Phi(S^3\times\{0\})=\Sigma$. 
Let $\Phi^\ast(\FF)$ be the pullback foliation on $S^3\times\R$ (recall that this is only a $C^0$ foliation) and set $\Sigma_t=\Phi_t(S^3)$. 
Since the cycle is a lacunary vanishing cycle contained in $L$ and $L\subset\lim_\mathbf{e}(L)$, it follows (possibly after reversing the sign of $\mathbb R$) that there exists a decreasing 
sequence $t_n\to 0$, $n\in\N$, so that each $\Sigma_{t_n}\subset L$ and $\Sigma_{t_n}$ bounds a manifold $C_n\subset  L$. Now $C_n$ must be simply connected 
by Van Kampen's Theorem, since both $S^3$ and $L$ are, so the manifolds $C_n$ lift to nearby leaves. By continuation from $C_1$, using 
Reeb stability, there exists a minimal $a\in\R\cup\{-\infty\}$ so that $\Sigma_{t}$ bounds a manifold homeomorphic to $C_1$ for all $a<t\leq t_1$.

If $a<0$ the lacunary vanishing cycle would be trivial, so $0\leq a<t_1$. Then by the preceding theorem with $B=S^3$ and $k=4$, $\Sigma_a$ will be contained in the compact boundary leaf of a generalized Reeb component and $L$ is an interior leaf of that component, so $L$ cannot meet $\Sigma_t$ for $t\leq a$, contradicting the hypothesis that $\Sigma_0\subset L$.
\end{proof}

\begin{prop}
Let $\FF$ be a codimension one $C^{1,0}$ foliation in a compact $5$-manifold $M$. If there exists a leaf $L$ of $\FF$ diffeomorphic to $Y\in {\mathcal Y}$, then $L$ is a proper leaf without holonomy. \label{p:proper}
\end{prop}

\begin{proof}
Since $L$ is simply connected, it is a leaf without holonomy. We also
observe that $L$ has a saturated neighborhood not meeting any compact leaves, since a limit leaf of compact leaves is compact (see \cite{Haefliger} or Theorem 6.1.1 in \cite{Candel-Conlon}).

First consider the case that $H_2(Y)\neq 0$. Let $K_L\subset L$ be a compact connected submanifold diffeomorphic to $K_Y$ (see notation \ref{n:K Y}). By Reeb stability there exists a neighborhood $U$ of $K_L$ bifoliated homeomorphically as a product. If $L$ meets $U$ in more than one connected component then there exists a compact subset $B\subset L$ homeomorphic to $K_L$ (via the transverse projection in $U$) and disjoint from $K_L$. This is impossible since the inclusion $i_0: K_L\hookrightarrow L$ induces an isomorphism
$i_{0*}: H_2(K_L)\to H_2(L)$ and the Mayer-Vietoris sequence applied to $(K_L\cup B) \bigcup (L\setminus (\mathring{K}_L\cup \mathring{B})$ shows that $B$ would give an additional non-trivial summand in $H_2(L)$. So in this case $L$ is a proper leaf.

Let us consider now the case where $Y\in\mathcal{R}\cup\overline{\mathcal{R}}$ is an exotic $\R^4$ (with finite Taylor index). Let $t>r_{\psi_\mathbf{R}}$ and let $D\subset L$ be a exotic disk diffeomorphic to $\mathbf{K}_t$. Since $\FF$ is $C^{1,0}$, by Reeb stability $D$ lifts to disks on nearby leaves such that their interiors are $C^1$-diffeomorphic to  $\mathring{\mathbf{K}}_t$.
If $L$ is non-proper then Reeb stability will produce infinitely many pairwise disjoint copies of $\mathbf{K}_t$ embedded in $Y$, so $\gamma(Y)=\infty$ by Proposition~\ref{p:TaylorR} which gives a contradiction.

Finally consider the case that $H_2(Y)=0$ and $Y$ has at least two ends (so it is homeomorphic to $\R^4$ with at least one puncture). Let $\mathbf{e}$ be an end of $Y$ diffeomorphic to the end of some $M\in\mathcal{M}$ and recall that for all $n\in\N$, $M=K_n\cup\mathbf{C}_n\cup X_n$ (see Remark~\ref{r:good-compact}), where $X_n$ is a neighborhood of the end of $M$, $K_n$ is a compact region and $\mathbf{C}_n$ is a topological compact cylinder bounding both manifolds. If $L$ is not proper,
since the number of ends of $Y$ is finite, there exists an end $\mathbf{e'}$ such that $L\subset\lim_{\mathbf {e'}}(L)$. 
By hypothesis, for some sufficiently large $k$, there exists a simply connected compact set $C$ diffeomorphic to $\mathbf{C}_{k}$ embedded in a neighborhood of $\mathbf{e}$ and separating the end $\mathbf{e}$ from the other ends of $L$, such that $\mathring{K}_k\in\mathcal{M}$. By Reeb stability there exists a neighborhood of $C$ bifoliated as a product. Thus there exists a smooth bifoliated embedding $j:C\times(-1,1)\to\FF$ so that $j(C\times\{0\}) = C$. As above, the projection of a tangential leaf to another in this neighborhood is a $C^1$ diffeomorphism. Thus $L\cap j(C\times(-1,1))$ contains a non-trivial sequence of tangential fibers $j(C\times\{s_m\})$, with ${s_m}$ tending to $0$, all of them contained in a neighborhood $X_\mathbf{e'}$ of the recurrent end $\mathbf{e'}$.

By Lemma~\ref{l:trivial lacunary}, for all sufficiently large $m$, $j(C\times \{s_m\})$ disconnects $\mathbf{e'}$ from the other ends. Otherwise infinitely many $j(C\times\{s_m\})$ would bound disks and $L$ would contain a non-trivial lacunary vanishing cycle homeomorphic to $S^3$. 

Note that $C$ has two boundary components, corresponding to $\partial^-\mathbf{C}_k$ and $\partial^+\mathbf{C}_k$. 
We say that $j(C\times \{s_m\})$ is {\em positively oriented\/} if $\mathbf{e'}$ is an end of the connected component of $L\setminus j(\mathring{C}\times \{s_m\})$ which contains the component of $\partial C$ corresponding to $\partial^+\mathbf{C}_k$; otherwise we say that it is {\em negatively oriented\/}. 

Let $s_1,s_2$ be distinct values in the sequence $\{s_m\}_m\in\N$ (possibly changing the subscripts), sufficiently close to $0$ so that both $j(C\times\{s_1\})$ and $j(C\times\{s_2\})$ disconnect $\mathbf{e'}$ from the other ends and such that both have the same orientation.
Let $N=j(\mathring{C}\times\{s_1\})\cup P\cup j(\mathring{C}\times\{s_2\})$, where $P$ is the connected component of $L\setminus \left(j(\mathring{C}\times\{s_1\})\cup j(\mathring{C}\times\{s_2\})\right)$ that meets both $C\times\{s_1\}$ and $C\times\{s_2\}$.
Thus $N$ is an exotic compact cylinder in $X_{\mathbf{e'}}$ that contains both $j(\mathring{C}\times\{s_1\})$ and $j(\mathring{C}\times\{s_2\})$.

We can assume that $j(C\times\{s_1\})$ contains the {\em negative\/} boundary of $N$, i.e. the boundary component that corresponds to $\partial^-\mathbf{C}_k$ and $j(C\times\{s_2\})$ contains the {\em positive\/} boundary of $N$, i.e. the boundary component that corresponds to $\partial^+\mathbf{C}_k$. Let us consider the following manifold: \[W=K_k\cup_{i_-}N\cup_{i}N\cup_{i}\dots\;,\] where $i_-$ is an orientation preserving diffeomorphism which maps $\mathring{\mathbf{C}}_k$ to $j(\mathring{C}\times\{s_1\})$, $i$ is another orientation preserving diffeomorphism from $j(\mathring{C}\times\{s_2\})$ to $j(\mathring{C}\times\{s_1\})$. Clearly $W\in\mathcal{M}$ since $\mathring{K}_k\in\mathcal{M}$. But this is in contradiction with Taubes' theorem, which implies that no element in $\mathcal{M}$ is smoothly periodic (Remark~\ref{r:TaubesM}).
\end{proof}

\begin{prop}\label{p:product_neighborhood}
If there exists a leaf $L$ diffeomorphic to $Y$, then there exists an open $\FF$-saturated neighborhood $U$ of $L$ which is diffeomorphic to $L\times(-1,1)$ by a diffeomorphism which carries the bifoliation $\FF$ and $\NN$ to the product bifoliation. In particular, all the leaves of $\FF_{|U}$ are diffeomorphic to $Y$.
\end{prop}
\begin{proof}
Since $L$ is a proper leaf, there exists a path, $c:[0,1)\to M$, transverse to $\FF$, with positive orientation and such that $L\cap c([0,1))=\{c(0)\}$. Let $U$ be the saturation of $c((0,1))$, which is a connected saturated open set and consider the octopus decomposition of $\hat{U}$ as described in Theorem~\ref{t:octopus}. Clearly one of the boundary leaves of $\hat{U}$ is diffeomorphic to $L$ because it is proper without holonomy and $c(0)\in L$. We identify this boundary leaf with $L$ and extend the nucleus $K$ so that the set $K'=\partial^\tau K\cap L$ is homeomorphic to $K_Y$. By Reeb stability, there exists a neighborhood of $K'$ foliated as a product by $K_Y\times\{\ast\}$. Since $L\subset\partial\hat U$ has an end, there is an arm $B_1$ that meets $L$. The corresponding $S_1$ is diffeomorphic to a neighborhood of an end and homeomorphic to $S^3\times (0,\infty)$, thus $B_1$ is foliated as a product (i. e., the suspension must be trivial). The union of
a smaller product neighborhood of $L\cap B_1$
and the product neighborhood of $K_Y$ meeting $L$ gives a product neighborhood on the positive side of $S_1\cup K_Y$. We can proceed in the same way for all the ends (which are finitely many), thus obtaining a product neighborhood on the positive side of $L\equiv K_Y\cup S_1\cup\dots\cup S_k$.

Proceeding in the same way on the negative side of $L$ we can find the desired product neighborhood of $L$. Each leaf is clearly diffeomorphic to $Y$ since the projection to $L$ along leaves
of $\NN$ is a local diffeomorphism and bijective by the product structure.
\end{proof}

Let $\Omega$ be the union of leaves diffeomorphic to $Y$. By the previous Proposition this is an open set on which the restriction $\FF_{|\Omega}$ is defined by a locally trivial fibration, so its leaf space is homeomorphic to a (possibly disconnected) $1$-dimensional manifold. Let $\Omega_1$ be one connected component of $\Omega$.

\begin{lemma}\label{l:non-compact}
The completed manifold $\hat{\Omega}_1$ is not compact.
\end{lemma}
\begin{proof}
First we note that $\partial\hat{\Omega}_1$ cannot be empty, for otherwise all the leaves would be diffeomorphic to $Y$, hence proper and non-compact. It is a well known fact (see, e.g., \cite{Candel-Conlon}) that a foliation in a compact manifold with all leaves proper must have a compact leaf, for every minimal set of such a foliation is a compact leaf.

Now suppose that $\hat{\Omega}_1$ is compact and let $L$ be a leaf diffeomorphic to $Y$.
Take an exotic end $\mathbf{e}$ of $L$ that is
end-diffeomorphic to an element in $\mathcal{M}$. Then the limit set of $\mathbf{e}$ of $L$ contains a minimal set, which must be contained in the boundary of $\hat{\Omega}_1$ and must be a compact leaf. The holonomy of the leaf $F$ is not the identity and has no fixed points  (otherwise it would produce non-trivial holonomy on an interior leaf). Since all the orbits are proper, the holonomy group of each boundary leaf must be isomorphic to $\Z$.

The contracting map that generates the holonomy of $F$ extends to a bifoliated map $h':X_1\to X_1$ that preserves each leaf of both $\FF$ and $\NN$ on a neighborhood $X_1$ of $F$ in $\hat{\Omega}_1$ (just by following the flow $\NN$ in the direction towards $F$). Since the holonomy is cyclic, each connected component of $L\cap X_1$ is end-diffeomorphic to an end of a cyclic covering space of $F$, so there exists an open neighborhood $V$ of $\mathbf{e}$ in $L$ where $h'$ is defined such that $V\subset X_1$ and $h=h'|_V$ is an embedding $h:V\to h(V)\subset V$ so that $\{h^n(V)\}\ (n\geq 0)$ is a neighborhood base of the end $\mathbf{e}$. But this contradicts Taubes' Theorem~\ref{t:Taubes} and Remark~\ref{r:TaubesM} which implies that $\mathbf{e}$ cannot be smoothly periodic.
\end{proof}

Following the approach of Ghys in \cite{Ghys1}, we have a dichotomy: the leaf space of $\FF_{|\Omega_1}$, which is a connected $1$-dimensional
manifold, must be either $\R$ or $S^1$.

\begin{prop}\label{p:non-R}
The leaf space of $\FF_{|\Omega_1}$ cannot be $\R$.
\end{prop}
\begin{proof}
Since $\hat{\Omega}_1$ is not compact there exists at least one arm for its octopus decomposition. Let $B_1$ be such an arm that is diffeomorphic to $S_1\times[0,1]$ via a diffeomorphism $\phi_1$ carrying the vertical foliation to $i^\ast\NN$. If the leaf space is $\R$, then $\phi_1(\{\ast\}\times(0,1))$ must meet each leaf in at most one point. Then the Trivialization Lemma~\ref{l:trivialization} shows that the saturation of $\phi_1(\{\ast\}\times(0,1))$ is diffeomorphic to a product $L\times(0,1)$.
Then the process of completing $\Omega_1$ to $\hat\Omega_1$
shows that the product $L\times(0,1)$ extends to a product
$L\times[0,1)$, so the boundary leaf corresponding to $L\times \{0\}$
must be diffeomorphic to $Y$, but this is a contradiction since leaves diffeomorphic to $Y$ have to be interior leaves of $\Omega$.
\end{proof}

Since $\Omega_1$ cannot fiber over the line, it must fiber over
the circle, but this is just the conclusion of Proposition~\ref{prop1}, so its proof is complete. The rest of this section is devoted to proving Theorem~\ref{mainthm}. It will be a quick corollary of Proposition~\ref{prop1} but first we have to introduce some terminology.

Let $U$ be a saturated open set of a $C^{1,0}$ foliation $\FF$ on a compact manifold and let $L\subset U$ be a leaf. For a Dippolito decomposition (Theorem~\ref{t:octopus}) $\hat{U}=K\cup B_1\cup\dots\cup B_n$, let us denote a complete transversal (homeomorphic to $[0,1]$) of each arm $B_i$ by $T_i$. 

\begin{defn}
Let $U$ be a saturated open set of a $C^{1,0}$ foliation $\FF$ on a compact manifold and let $L\subset U$ be a leaf. Then $L$ is said to be {\em trivial at infinity} for $U$ if there exists a Dippolito decomposition $\hat{U}=K\cup B_1\cup\dots\cup B_n$, and total transversals $T_i$ of each $B_i$, such that $L\cap T_i$ consists of fixed points for every element of the total holonomy group associated to that arm.
\end{defn}

The next theorem is a consequence of the so called {\em generalized Kopell lemma\/} for foliations which can be found in \cite{Cantwell-Conlon2}
and in (\cite{Candel-Conlon}, Proof of Theorem 8.1.26). We thank the first referee for suggesting this argument.

\begin{thm}\cite{Cantwell-Conlon2}\label{t:Cantwell-Conlon2}
Let $\FF$ be a transversely oriented codimension one $C^2$ foliation on a compact manifold.  If $L$ is a proper leaf then it is trivial at infinity for every saturated open set containing $L$.
\end{thm}

From this theorem we can deduce the following.

\begin{lemma} \label{l:infinitely many ends}
Suppose $\FF$ is a transversely oriented codimension one $C^{2}$ foliation on a compact manifold. Let $U$ be an open saturated set. If  $\hat{U}$ is non-compact and $L\subset U$ is a proper leaf such that $\lim L $ contains a non-compact manifold of $\partial\hat{U}$, then $L$ has infinitely many ends.\end{lemma}
\begin{proof}
By the above Theorem~\ref{t:Cantwell-Conlon2}, $L$ is trivial at infinity, so there exists a Dippolito decomposition $\hat{U}=K\cup B_1\cup\dots\cup B_n$ such that all the points of $L\cap T_i\neq\emptyset$ are fixed points for the total holonomy group in the arm $B_i$ for every $i$. Since $\lim(L)$ contains a non-compact boundary leaf, some $L\cap T_i$ must consist of infinitely many points, since it contains a discrete set converging to a boundary leaf. Since all of these points are fixed, it follows that $L\cap B_i$ is a disjoint union of infinitely many copies of the base manifold of that arm. Each one of these copies defines a different end for the leaf $L$. 
\end{proof}

\subsection{Proof of Theorem~\ref{mainthm}}
Let $L$ be a leaf diffeomorphic to some $Y\in\mathcal{Y}$ and suppose that $\FF$ is $C^2$. By Proposition~\ref{prop1}, $L$ is proper and the set $\Omega_1$ (the connected component of the union of the leaves diffeomorphic to $L$ that contains $L$) is an open, connected, and saturated set fibering over the circle. Then the limit of $L$ in $\hat{\Omega}_1$ is non-empty (by Proposition~\ref{p:non-R} the total holonomy group on each arm is non trivial, so transverse accumulation points do exist) and consists of non-compact leaves of $\partial\hat{\Omega}_1$ (by the same argument as in the proof of Lemma~\ref{l:non-compact}). Thus the hypothesis of Lemma~\ref{l:infinitely many ends} is satisfied and $L$ must have infinitely many ends in contradiction to the fact that every manifold in $\mathcal{Y}$ has finitely many ends.
\qed


\section{Final comments}

It is possible to enlarge the family $\mathcal{Y}$ if, instead of working with ends homeomorphic to $S^3\times [0,\infty)$, we allow admissible ends (in the sense of Definition 1.3 in \cite{Taubes}) that are simply connected. We avoid working with this generality in the interests of readability.

Recall that Proposition~\ref{prop1} says that if $Y\in\mathcal Y$ is diffeomorphic to a leaf, then it is a proper leaf without holonomy contained in an open saturated set which consists of leaves diffeomorphic to $Y$ and fibers over the circle. Let us consider as in \cite{Ghys1} the map $h:\Omega_1\to\Omega_1$ which maps each point $x\in L'\subset\Omega_1$ to the first return point $h(x)\in L'$ along the transverse foliation $\NN$ in the negative direction. This is well defined globally because each leaf has a neighborhood bifoliated as a product and the leaf space is the circle. This is the monodromy map which is an orientation-preserving automorphism of our exotic manifold $L$. We saw in the description of exotic structures that self-diffeomorphisms of $\mathbf{R}\in \mathcal{R}\cup\overline{\mathcal{R}}$ are in some sense rigid; this allows us to say what kind of monodromy is admissible at the present state of the art. For instance, the compact set $\mathbf{K}_t\subset\mathbf{R}$ with $t>r_{\psi_{\mathbf{R}}}$ must meet the nucleus of any octopus decomposition of $\Omega_1$ and it seems likely that the monodromy should have fixed points.

As far as we know, this work gives the first insight into the problem of realizing exotic structures on open $4$-manifolds as leaves of a foliation in a compact manifold. We express our hopes in the following conjecture, which we are far from proving, since it includes the higher codimension case and lower regularity assumptions, which are not treated in this paper. It is a goal for future research.

\begin{conj}
No open $4$-manifold with an isolated Taubes like end is diffeomorphic to a leaf of a $C^{1,0}$ foliation of arbitrary codimension in a compact manifold. \label{conj}
\end{conj}

The corresponding conjecture is in fact open for all the known families (of ends) of exotic $\R^4$'s but, as a consequence of this work, those which are Taubes like are the best candidates. Another interesting question is what can be said if we allow infinitely many ends, since in this case Lemma~\ref{l:infinitely many ends} will not work. For instance $S^4$ minus a Cantor set is diffeomorphic to a leaf of a $C^\omega$ codimension $1$ compact manifold (observe that it is the universal covering space of a compact smooth $4$-manifold with fundamental group isomorphic to $\Z\ast\Z$, so any suspension of an analytic action of that group on $S^1$ does the job\footnote{there are lots of such actions, see e.g. \cite{Navas}.}).
It would be interesting to exhibit smooth structures on this manifold which are not diffeomorphic to leaves. This question will be treated in a forthcoming work of the authors.

Let us say something about small exotica (those that embed as open sets in the standard $\R^4$). Small exotica are more interesting from a physical point of view since they support Stein structures (see e.g. \cite{Gompf-Stipsicz}). There is a Taubes-type theorem for them based on the work of DeMichelis and Freedman \cite{DeMichelis-Freedman} and with more generality in \cite{Taylor2}, but unfortunately it does not seem to be enough to prove a version of Lemma~\ref{l:non-compact}. In addition, there is no known ``Taylor index" invariant and therefore the first part of our arguments, which shows that the leaf must be a proper leaf, fails for small exotica, although it works for punctured simply connected $4$-manifolds obtained by removing finitely many points from closed manifolds not homeomorphic to $S^4$, since for these manifolds the argument is purely topological.

It is also worth noting that if the smooth $4$-dimensional Poincar\'e conjecture is false then it is easy to produce exotic $\R^4$'s which are leaves of a transversely analytic (in particular  $C^{2}$) foliation. Consider
$S^4 \times S^1$ with the product foliation, where $S^4$ has an exotic smooth structure, and insert a Reeb component along a
transverse curve, for example $\{*\}\times S^1$. This can easily be done so as to preserve the transverse analyticity. The leaves would be exotic $\R^4$'s with a standard smooth structure at the end.

Finally we include a last remark. Recent work of J. \'Alvarez L\'opez and R. Barral Lij\'o \cite{Alvarez-Barral} states that every Riemannian manifold with bounded geometry can be realized isometrically as a leaf in a compact foliated space. It is known \cite{Greene} that every smooth manifold supports such a geometry, so it follows as a corollary that every smooth manifold is diffeomorphic to a leaf in a compact foliated space. In particular this holds for any exotic $\R^4$. However the transverse topology of this foliated space would in general be far from being a manifold. Anyway, this gives us some hope of finding an explicit description of exotic structures by using finite data: the tangential change of coordinates of a finite foliated atlas.

\end{document}